\documentclass{amsart}
\usepackage{amsmath,amsfonts,latexsym,amssymb}

\newtheorem{theorem}{Theorem}
\newtheorem{lemma}{Lemma}

\newtheorem{proposition}{Proposition}

\theoremstyle{remark}

\begin{document}

\markboth{Ritabrata Munshi}{Subconvexity for twists of $GL(3)$ $L$-functions}
\title[Subconvexity for twists of $GL(3)$ $L$-functions]{The circle method and bounds for $L$-functions - II: Subconvexity for twists of $GL(3)$ $L$-functions}

\author{Ritabrata Munshi}   
\address{School of Mathematics, Tata Institute of Fundamental Research, 1 Dr. Homi Bhabha Road, Colaba, Mumbai 400005, India.}     
\email{rmunshi@math.tifr.res.in}

\begin{abstract}
Let $\pi$ be a $SL(3,\mathbb Z)$ Hecke-Maass cusp form. Let $\chi=\chi_1\chi_2$ be a Dirichlet character with $\chi_i$ primitive modulo $M_i$. Suppose $M_1$, $M_2$ are primes such that $\sqrt{M_2}M^{4\delta}<M_1<M_2M^{-3\delta}$, where $M=M_1M_2$ and $0<\delta<1/28$. In this paper we will prove the following subconvex bound
$$
L\left(\tfrac{1}{2},\pi\otimes\chi\right)\ll_{\pi,\varepsilon} M^{\frac{3}{4}-\delta+\varepsilon}.
$$
\end{abstract}


\maketitle


\section{Introduction}
\label{intro}

Let $\pi$ be a  Hecke-Maass cusp form for $SL(3,\mathbb Z)$ with normalized Fourier coefficients $\lambda(m,n)$. Let $\chi$ be a primitive Dirichlet character modulo $M$. The twisted $L$-series $L(s,\pi\otimes\chi)$, which is given by 
$$
\sum_{n=1}^\infty \lambda(1,n)\chi(n)n^{-s}
$$
in the domain $\sigma=\text{Re}(s)>1$, extends to an entire function and satisfies a functional equation with arithmetic conductor $M^3$. The subconvexity problem for this $L$-function has recently drawn much attention, and has been solved in several special cases in \cite{B}, \cite{Mu1}, \cite{Mu2} and \cite{Mu3}. (Also see \cite{L} for $t$-aspect subconvexity.) However, all these works suffer from a major drawback that they deal only with forms which are (symmetric square) lifts of $GL(2)$ forms. Subconvexity for twists of `genuine' $GL(3)$ forms still remains untouched. In this paper we make a modest attempt to bridge the gap by proving the following.      
\begin{theorem}
\label{mthm}
Let $\pi$ be a  Hecke-Maass cusp form for $SL(3,\mathbb Z)$, and let $\chi=\chi_1\chi_2$ be a Dirichlet character with $\chi_i$ primitive modulo $M_i$. Suppose $M_1$, $M_2$ are primes such that $\sqrt{M_2}M^{4\delta}<M_1<M_2M^{-3\delta}$, where $M=M_1M_2$ and $0<\delta<1/28$. Then we have
$$
L\left(\tfrac{1}{2},\pi\otimes\chi\right)\ll_{\pi,\varepsilon} M^{\frac{3}{4}-\delta+\varepsilon}.
$$
\end{theorem}

The conductor of the $L$-function is of size $M^3$ (ignoring the dependence on the form $\pi$). Hence the convexity bound is given by $L\left(\tfrac{1}{2},\pi\otimes\chi\right)\ll M^{3/4}$. The result is similar to the one proved in \cite{Mu2} for the symmetric square $L$-function. But the method used here is different and is based on the ideas introduced in \cite{Mu}. In particular we do not compute moment. We use the circle method, with a `conductor lowering' trick, directly to the approximate functional equation as a device for separation of oscillation. Then we use Poisson summation and $GL(3)$ Voronoi summation formula. Even after getting square root cancellation in the resulting complete character sums, we observe that we are still short of convexity. Also summing over the moduli seems quite difficult. Instead we apply Cauchy to escape from the `trap of involution', and then apply Poisson summation formula once again. This leads us to complicated mixed character sums. Square root cancellation for such sums follows from the work of Deligne. This gives us enough savings to break the convexity barrier. Finally we note that the assumption that $M_i$ are primes is only a technical convenience.\\

We use a similar approach in the companion paper \cite{Mu0} to establish $t$-aspect subconvexity for $SL(3,\mathbb Z)$ Hecke-Maass forms. There we use an archimedean analogue of our conductor lowering trick, and use stationary phase method (in place of Deligne's bound) to get `square-root' cancellation in certain exponential integrals.


\section{Preliminaries}
\label{prelim}

Suppose $\pi$ is a Maass form of type $(\nu_1,\nu_2)$ for $SL(3,\mathbb Z)$ which is an eigenfunction of all the Hecke operators with Fourier coefficients $\lambda(m_1, m_2)$, normalized so that $\lambda(1, 1)=1$ (for details see Goldfeld's book \cite{G}). We introduce the Langlands parameters $(\alpha_1, \alpha_2, \alpha_3)$, defined by $\alpha_1=-\nu_1-2\nu_2+1$, $\alpha_2=-\nu_1+\nu_2$ and $\alpha_3=2\nu_1+\nu_2-1$. The Ramanujan-Selberg conjecture predicts that $\text{Re}(\alpha_i)=0$, and from the work of Jacquet-Shalika we (at least) know that $|\text{Re}(\alpha_i)|<\frac{1}{2}$. \\

The Voronoi summation formula (see \cite{L}, \cite{MS}) will play a crucial role in our analysis. Let $g$ be a compactly supported function on $(0,\infty)$, and let $\tilde g(s)=\int_0^\infty g(x)x^{s-1}dx$ be the Mellin transform. For $\ell=0,1$ define
\begin{align}
\label{gammma-factor}
\gamma_{\ell}(s)=\frac{1}{2\pi^{3(s+\frac{1}{2})}}\prod_{i=1}^3\frac{\Gamma\left(\frac{1+s+\alpha_i+\ell}{2}\right)}{\Gamma\left(\frac{-s-\alpha_i+\ell}{2}\right)}
\end{align}
and set $\gamma_\pm(s)=\gamma_0(s)\mp i\gamma_1(s)$. For $g$ as above we define the integral transforms 
\begin{align}
\label{gl}
G_{\pm}(y)=\frac{1}{2\pi i}\int_{(\sigma)}y^{-s}\gamma_{\pm}(s)\tilde g(-s)\mathrm{d}s
\end{align}
where $\sigma>-1+\max\{-\text{Re}(\alpha_1),-\text{Re}(\alpha_2),-\text{Re}(\alpha_3)\}$. Let  
$$
S(a,b;c)=\sideset{}{^\star}\sum_{\alpha\bmod{c}}e\left(\frac{a\alpha+b\overline{\alpha}}{c}\right)
$$ 
be the Kloosterman sum, where $\bar{\alpha}$ denotes the multiplicative inverse of $\alpha\bmod{c}$ and $e(z)=e^{2\pi iz}$.
 
\begin{lemma}
Let $g$ be a compactly supported function on $(0,\infty)$, we have
\begin{align}
\label{voronoi3}
\sum_{n=1}^\infty \lambda(1,n)e\left(\frac{an}{q}\right)g(n)=&q\sum_\pm \sum_{n_1|q}\sum_{n_2=1}^\infty \frac{\lambda(n_2,n_1)}{n_1n_2}S(\bar a, \pm n_2; q/n_1)G_\pm\left(\frac{n_1^2n_2}{q^3}\right),
\end{align}
where $(a,q)=1$ and $\bar{a}$ denotes the multiplicative inverse of $a\bmod{q}$. 
\end{lemma}

The following lemma is also well-known (see \cite{Mol}).
\begin{lemma}
\label{ram-on-av}
We have
$$
\mathop{\sum\sum}_{n_1^2n_2\leq x}|\lambda(n_1,n_2)|^2\ll x^{1+\varepsilon},
$$
where the implied constant depends on the form $\pi$ and $\varepsilon$.\\
\end{lemma}

Finally let us recall that the approximate functional equation implies that
\begin{align}
\label{afe}
L\left(\tfrac{1}{2},\pi\otimes\chi\right)\ll M^{\varepsilon}\sup_{N\leq M^{3/2+\varepsilon}} \frac{S(N)}{\sqrt{N}}
\end{align}
where $S(N)$ are sums of the type 
$$
S(N):=\sum_{n=1}^\infty \lambda(1,n)\chi(n)V\left(\frac{n}{N}\right)
$$
for some smooth function $V$ supported in $[1,2]$ and satisfying $V^{(j)}(x)\ll_j 1$. Hence to establish subconvexity we need to show cancellation in the sum for $N$ of size $M^{3/2}$, roughly speaking.

\vspace{.5cm}

We will now briefly recall a version of the circle method. In the previous paper \cite{Mu} in this series we used Jutila's version of the circle method with factorizable moduli to gain structural advantage. In the present case Kloosterman's version of the circle method works better. (One can also use $\delta$-symbol method.) Let
$$
\delta(n)=\begin{cases} 1&\text{if}\;\;n=0;\\
0&\text{otherwise}.\end{cases}
$$
\begin{lemma}\label{delta-symbol}
Let $Q$ be a positive real number. Then for any integer $n\in\mathbb Z$ we have
\begin{align}
\label{cm}
\delta(n)=2\:\text{Re}\int_0^1\mathop{\sum\sideset{}{^\star}\sum}_{1\leq q\leq Q <a\leq q+Q}\frac{1}{aq}e\left(\frac{n\bar a}{q}-\frac{nx}{aq}\right)\mathrm{d}x.
\end{align}
The $\star$ on the sum indicates that the sum over $a$ is restricted by the condition $(a,q)=1$. 
\end{lemma}
(For a proof of this formula see \cite{IK}.) There are well understood drawbacks in this form of circle method. However in our treatment these do not create any problem as we will not need to execute the complete character sum over $a$. We will only need the fact that $a\asymp Q$. (The notation $\alpha\asymp A$ means that there exists absolute constants $0<c_1<c_2$ such that $c_1A< |\alpha| <c_2A$.)

\section{Application of circle method}
\label{setup}

We will establish the following. 
\begin{proposition}
\label{prop1}
Suppose $M_i$ are as in Theorem~\ref{mthm}. We have
\begin{align}
\label{prop-eq}
S(N)\ll N^{3/4}\left(M_1^{3/4}+\sqrt{M_2}\right)M^\varepsilon+\frac{N\sqrt{M_2}}{M_1}M^\varepsilon.
\end{align}
\end{proposition}
For $N\leq M_2^2$ the trivial bound $S(N)\ll NM^\varepsilon$, which follows from Lemma \ref{ram-on-av}, is sufficient for our purpose as in this case $N\leq N^{3/4}\sqrt{M_2}$. Clearly Theorem \ref{mthm} follows from the bound given in \eqref{prop-eq} and \eqref{afe}. In the rest of the paper we will prove the proposition for $N> M_2^2$ .\\

As in \cite{Mu} we will apply the circle method directly to the smooth sum  $S(N)$ which appears in \eqref{afe}. But instead of doing it in one step, we will break it up into two steps using the modulus $M_1$, viz.
$$
S(N)=\mathop{\sum\sum}_{\substack{n,m=1\\M_1|n-m}}^\infty \lambda(1,n)\chi(m)\:\delta\left(\frac{n-m}{M_1}\right)V\left(\frac{n}{N}\right)V^\star\left(\frac{n-m}{N}\right).
$$
Here $V^\star$ is an even smooth function supported in $[-1,1]$, and $V^\star(0)=1$, $V^{\star (j)}\ll_j 1$. The (integral) equation $n-m=0$ is equivalent to the congruence $n-m\equiv 0\bmod{M_1}$ and the (integral) equation $(n-m)/M_1=0$. This `factorization' process acts like a conductor lowering mechanism, as the modulus $M_1$ is already present in the character $\chi$. A similar idea was used in a different vein in \cite{BM}.\\

Applying Lemma \ref{delta-symbol}, and choosing $Q=\sqrt{N/M_1}$, we get 
$$
S(N)=S^+(N)+S^-(N),
$$
where 
$$
S^\pm (N)=\mathop{\sum\sideset{}{^\star}\sum}_{1\leq q\leq Q <a\leq q+Q}\frac{1}{aq}\mathop{\sum\sum}_{\substack{n,m=1\\M_1|n-m}}^\infty \lambda(1,n)\chi(m)e\left(\pm\frac{\bar a(n-m)/M_1}{q}\right)W\left(\frac{n}{N},\frac{m}{N},\frac{\pm N}{M_1aq}\right),
$$
with
$$
W(u,v,y)=\int_0^1e\left((v-u)xy\right)V\left(u\right)V^\star\left(u-v\right)\mathrm{d}x.
$$ 
For notational simplicity we will only analyse $S^+(N)$. The analysis of $S^-(N)$ is just similar. We further approximate $S^+(N)$ by
$$
\tilde S(N):=\mathop{\sum\sideset{}{^\star}\sum}_{\substack{1\leq q\leq Q <a\leq q+Q\\(q,M_1)=1}}\frac{1}{aq}\mathop{\sum\sum}_{\substack{n,m=1\\M_1|n-m}}^\infty \lambda(1,n)\chi(m)e\left(\frac{\overline{aM_1}(n-m)}{q}\right)W\left(\frac{n}{N},\frac{m}{N},\frac{N}{M_1aq}\right).
$$\\

\begin{lemma}
\label{circ}
We have
$$
S^+(N)=\tilde S(N)+O_{\pi,\varepsilon}\left(\frac{N\sqrt{M_2}}{M_1}M^{\varepsilon}\right).
$$
\end{lemma}
\begin{proof}
We want to estimate
$$
E=\mathop{\sum\sum}_{\substack{1\leq q\leq Q/M_1\\ Q<a\leq qM_1+Q\\(a,qM_1)=1}}\frac{1}{aqM_1}\mathop{\sum\sum}_{\substack{n,m=1\\M_1|n-m}}^\infty \lambda(1,n)\chi(m)e\left(\frac{\bar a(n-m)/M_1}{qM_1}\right)W\left(\frac{n}{N},\frac{m}{N},\frac{N}{M_1^2aq}\right).
$$
In the sum, $q$ ranges up to $Q/M_1\leq M_1^{-3/4}M_2^{3/4}M^{\varepsilon} \leq  M_1^{3/4}M^{\varepsilon}$, where the last inequality follows from our assumption regarding the sizes of $M_i$. Hence $(q,M)=1$. We apply Poisson summation to the sum over $m$ with modulus $M_1^2M_2q$. This yields
$$
E=\frac{N}{M_1^3M_2}\mathop{\sum\sum}_{\substack{1\leq q\leq Q/M_1\\ Q<a\leq qM_1+Q\\(a,qM_1)=1}}\frac{1}{aq^2}\mathop{\sum\sum}_{\substack{n,m\in\mathbb Z\\n>0}} \lambda(1,n)e\left(\frac{\bar an}{qM_1^2}\right)C(m,q)\int_{\mathbb R}W\left(\frac{n}{N},y,\frac{N}{M_1^2aq}\right)e\left(\frac{-mNy}{M_1^2M_2q}\right)\mathrm{d}y.
$$
where
$$
C(m,q)=\sum_{\substack{b \bmod{M_1^2M_2q}\\b\equiv n\bmod{M_1}}}\chi(b)e\left(\frac{-\bar ab}{qM_1^2}\right)e\left(\frac{mb}{M_1^2M_2q}\right).
$$
Using the coprimality of $M_1$, $M_2$ and $q$, we evaluate the character sum, and observe that $C(m,q)=0$ unless $am\equiv M_2\bmod{M_1q}$, in which case we have $|C(m,q)|\ll qM_1\sqrt{M_2}$. The integral on the other hand is negligibly small if $|m|\gg M^{1+\varepsilon}Q/N$ (which follows from repeated integration by parts), otherwise it is bounded by $O(1)$. So it follows that 
$$
E\ll \frac{N}{M_1^3M_2Q}\sum_{q\leq Q/M_1}\sum_{n\leq 2N}\frac{|\lambda(1,n)|}{q^2}\mathop{\sum}_{1\leq |m|\ll M^{1+\varepsilon}Q/N}\;qM_1\sqrt{M_2} \ll \frac{N\sqrt{M_2}}{M_1}M^{\varepsilon}.
$$ 
The last inequality follows from Lemma~\ref{ram-on-av}.
\end{proof}

Next we detect the congruence $n\equiv m\bmod{M_1}$ in the definition of $\tilde S(N)$ using exponential sums to obtain
$$
\frac{1}{M_1}\mathop{\sum\sideset{}{^\star}\sum}_{\substack{1\leq q\leq Q <a\leq q+Q\\(q,M_1)=1}}\frac{1}{aq}\sum_{b\bmod{M_1}}\mathop{\sum\sum}_{\substack{n,m=1}}^\infty \lambda(1,n)\chi(m)e\left(\frac{(\overline{aM_1}M_1+bq)(n-m)}{qM_1}\right)W\left(\frac{n}{N},\frac{m}{N},\frac{N}{M_1aq}\right).
$$
We split  this as 
$$
\tilde S(N)=\tilde S_0(N)+T(N)
$$
where
$$
\tilde S_0(N)=\frac{1}{M_1}\mathop{\sum\sideset{}{^\star}\sum}_{\substack{1\leq q\leq Q <a\leq q+Q\\(q,M_1)=1}}\frac{1}{aq}\mathop{\sum\sum}_{\substack{n,m=1}}^\infty \lambda(1,n)\chi(m)e\left(\frac{\overline{aM_1}(n-m)}{q}\right)W\left(\frac{n}{N},\frac{m}{N},\frac{N}{M_1aq}\right)
$$
is the contribution of $b=0$.\\

In the rest of the paper we will analyse $T(N)$, and we will show that the bound \eqref{prop-eq} holds for $T(N)$. The other term $\tilde S_0(N)$ can be analysed in a similar fashion, and can be shown to be of much smaller magnitude. Indeed the trivial bound for this sum is $O(N^2M^{\varepsilon}/M_1)$. Applying the Poisson summation formula on $m$ (and evaluating the character sum) we are able to save $N/\sqrt{M}$. The Voronoi summation formula on $n$ gives us a saving of $N/Q^{3/2}$. This gives us a total saving of $N^{5/4}M_1^{1/4}M_2^{-1/2}$. So we get the following
\begin{lemma}
\label{lemma-s0}
We have
$$
\tilde S_0(N)\ll N^{3/4}M^{\varepsilon}\frac{\sqrt{M_2}}{M_1^{5/4}}.
$$
\end{lemma}
The detailed proof of this lemma can be given following our analysis in Section \ref{pvs}. The above bound is smaller than the second term on the right hand side of \eqref{prop-eq}.


\section{Poisson and Voronoi summation}
\label{pvs}

Consider
$$
T(N)=\frac{1}{M_1}\mathop{\sum\sideset{}{^\star}\sum}_{\substack{1\leq q\leq Q <a\leq q+Q\\(q,M_1)=1}}\frac{1}{aq}\;\sideset{}{^\star}\sum_{b\bmod{M_1}}T(a,b;q)
$$
where
\begin{align}
\label{tab}
T(a,b;q)=\mathop{\sum\sum}_{\substack{n,m=1}}^\infty \lambda(1,n)\chi(m)e\left(\frac{(\overline{aM_1}M_1+bq)(n-m)}{qM_1}\right)W\left(\frac{n}{N},\frac{m}{N},\frac{N}{M_1aq}\right).
\end{align}
In this section we will analyse the double sum $T(a,b;q)$ where $(a,q)=(b,M_1)=1$ (consequently $(\overline{aM_1}M_1+bq,qM_1)=1$).\\ 

\begin{lemma}
\label{poisson-voronoi-lemma}
We have
$$
T(a,b;q)=\varepsilon_1\varepsilon_2\left[T_+(a,b;q)+T_-(a,b;q)\right]
$$
where $\varepsilon_i\sqrt{M_i}$ is the value of the Gauss sum corresponding to $\chi_i$, and 
\begin{align*}
T_\pm(a,b;q)=qN\sqrt{\frac{M_1}{M_2}}\sum_{n_1|qM_1}\sum_{n_2=1}^\infty &\sum_{\substack{m\in\mathbb Z\\am\equiv M_2\bmod{q}}} \frac{\lambda(n_2,n_1)}{n_1n_2}\bar\chi_2(\overline{qM_1}m)\bar\chi_1(\overline{qM_2}m-b)\\
\times &S\left(\overline{\overline{aM_1}M_1+bq}, \pm n_2; qM_1/n_1\right)\mathcal I_{\pm,a} (n_1^2n_2,m;q).
\end{align*}
The integral is given by
$$
\mathcal I_{\pm,a} (n,m;q)=\frac{1}{2\pi i}\int_{(\sigma)}\left(\frac{q^3M_1^3}{nN}\right)^{s}\gamma_{\pm}(s)\hat W\left(s,m,\frac{N}{M_1aq}\right)\mathrm{d}s
$$
where
\begin{align*}
\hat W\left(s,m,\frac{N}{M_1aq}\right):=\mathop{\iint}_{\mathbb R^2}W\left(u,v,\frac{N}{M_1aq}\right)u^{-s-1}e\left(-\frac{mNv}{Mq}\right)\mathrm{d}u\mathrm{d}v.
\end{align*}
\end{lemma}

\begin{proof} 
Let $a'\equiv \overline{aM_1}\bmod{q}$. Applying Poisson on the sum over $m$ in \eqref{tab} we get
$$
\frac{N}{Mq}\mathop{\sum}_{n=1}^\infty \lambda(1,n)e\left(\frac{(a'M_1+bq)n}{qM_1}\right)\sum_{m\in\mathbb Z}\mathcal C(a,b;m,q)\int_{\mathbb R}W\left(\frac{n}{N},v,\frac{N}{M_1aq}\right)e\left(-\frac{mNv}{Mq}\right)\mathrm{d}v
$$
where
$$
\mathcal C(a,b;m,q)=\sum_{c\bmod{qM}}\chi(c)e\left(-\frac{(a'M_1+bq)c}{qM_1}+\frac{cm}{qM}\right),
$$
Using the coprimality of $M_1$, $M_2$ and $q$, we get 
\begin{align}
\label{char-exp}
\mathcal C(a,b;m,q)=\begin{cases}\varepsilon_1\varepsilon_2q\sqrt{M}\bar\chi_2(m)\chi_2(qM_1)\bar\chi_1(\overline{qM_2}m-b)&\text{if}\;\;am\equiv M_2\bmod{q};\\
0&\text{otherwise}.\end{cases}\end{align}
Here $\varepsilon_i\sqrt{M_i}$ is the value of the Gauss sum corresponding to the character $\chi_i$. In particular we should have $(m,q)=1$ for the sum not to vanish. The lemma now follows by applying the Voronoi summation formula, i.e. Lemma~\ref{voronoi3}, on the sum over $n$. \end{proof}

Now we will study the integrals $\mathcal I_{\pm,a} (n,m;q)$. 
\begin{lemma}
\label{int-trns-bd}
The integrals $\mathcal I_{\pm,a} (n,m;q)$ are negligibly small if $n\gg Q^3M_1^3M^\varepsilon/N$ or if $|m|\gg QM^{1+\varepsilon}/N$. Otherwise we have the bound
$$
\mathcal I_{\pm,a} (n,m;q)\ll \sqrt{\frac{nN}{q^2QM_1^3}}M^\varepsilon.
$$ 
Also we have
$$
u_2^j\frac{\partial^j}{\partial u_2^j}\mathcal I_{\pm,a} (n_1^2u_2,m;q)\ll \left(\frac{Q}{q}\right)^j\sqrt{\frac{n_1^2u_2N}{q^2QM_1^3}}M^\varepsilon.
$$ 
\end{lemma}
\begin{proof}
Replacing the expression for $W$ and making the change of variables $(u,v)\mapsto (u,w):=(u,v-u)$ (and using the fact that $V^\star$ is an even functions) we get
\begin{align*}
\hat W\left(s,m,\frac{N}{M_1aq}\right)=\left[\mathop{\int}_0^1\mathop{\int}_{\mathbb R}V^\star(w)e\left(\frac{Nxw}{M_1aq}\right)e\left(-\frac{mNw}{Mq}\right)\mathrm{d}w\mathrm{d}x\right]\left[\mathop{\int}_{\mathbb R}V(u)u^{-s-1}e\left(-\frac{mNu}{Mq}\right)\mathrm{d}u\right].
\end{align*}
Using repeated integration by parts, we see that the first double integral (and hence the integral $\mathcal I_{\pm,a} (n,m;q)$) is negligibly small if $|m|\gg M^{1+\varepsilon}Q/N$. For smaller values of $|m|$ we can bound this integral by $O(qM^\varepsilon/Q)$. To get this bound we integrate over $x$ for $|w|\geq M^{-2012}$, and then estimate the remaining integrals trivially. \\

Now consider the integral
\begin{align}
\label{one-more-int}
\frac{1}{2\pi i}\int_{(\sigma)}\left(\frac{q^3M_1^3}{nN}\right)^{s}\gamma_{\pm}(s)\mathop{\int}_{\mathbb R}V(u)u^{-s-1}e\left(-\frac{mNu}{Mq}\right)\mathrm{d}u\mathrm{d}s.
\end{align}
By repeated integration by parts in the inner integral, using Stirling's approximation and the fact that $|m|\ll M^{1+\varepsilon}Q/N$, and moving the contour to the right we get that the integral is negligibly small if $|n|\gg Q^3M_1^3M^{\varepsilon}/N$. For smaller values of $n$ we shift the contour to $\sigma=-1/2$. Let $s=-1/2+i\tau$. Then the inner integral in \eqref{one-more-int} is negligibly small unless $|\tau|\ll QM^{\varepsilon}/q$ as $m$ is such that $|m|\ll M^{1+\varepsilon}Q/N$. This follows from repeated integration by parts. For smaller values of $\tau$ we use the second derivative bound for the inner exponential integral and conclude that the integral in \eqref{one-more-int} is bounded by
$$
\ll \sqrt{\frac{nN}{q^3M_1^3}}\sqrt{\frac{Q}{q}}M^{\varepsilon}.
$$
This concludes the proof of the first two statements. For the third statement we use a similar analysis together with differentiation under the integral sign. 
\end{proof}

The contribution of $T_\pm(a,b;q)$ to $T(N)$ is given by 
\begin{align*}
T_\pm (N)=&\varepsilon_1\varepsilon_2\frac{N}{\sqrt{M}}\sum_{\substack{q\leq Q\\(q,M_1)=1}}\mathop{\sum\sum}_{\substack{n_1|q\\\ell|M_1}}\sum_{n_2=1}^\infty\sum_{(m,q)=1} \frac{1}{a(m,q)}\frac{\lambda(n_2,\ell n_1)}{\ell n_1n_2}\bar\chi_2(\overline{qM_1}m)\\
&\times S(M_2\overline{m\hat M_1}, \pm n_2\overline{\hat M_1}; \hat q)\left[\mathop{\sideset{}{^{\star}}\sum}_{ b\bmod{M_1}}\bar\chi_1(\overline{qM_2}m-b)S(\overline{bq\hat q}, \pm n_2\hat q; \hat M_1)\right]\mathcal I_\pm(\ell^2 n_1^2n_2,m;q)
\end{align*}
where $\hat q=q/n_1$, $\hat M_1=M_1/\ell$. Also $a(m,q)$ is the unique solution of the congruence $am\equiv M_2\bmod{q}$ in the range $Q<a\leq q+Q$, and $\mathcal I_\pm(n,m;q)$ is the shorthand for $\mathcal I_{\pm,a(m,q)}(n,m;q)$. Since $M_1$ is a prime, either $\ell=1$ or $\ell=M_1$. Accordingly we have a decomposition
$$
T_\pm (N)=T_{\pm,1} (N)+T_{\pm,M_1} (N).
$$
In the second case (i.e. $\ell=M_1$) the trivial bound is already satisfactory for our purpose. Indeed using Hecke relation, Lemma \ref{ram-on-av}, Lemma \ref{int-trns-bd} and Weil bound for Kloosterman sums, we see that this term is dominated by $O(N^{3/4}M_2^{1/4}/M_1)$. This is again smaller than the second term in the right hand side of \eqref{prop-eq}. \\

Now we consider the remaining case $\ell=1$. We have
\begin{align}
\label{T+}
T_{\pm,1} (N)\ll \sum_{L\;\text{dyadic} \ll Q^3M_1^3M^\varepsilon/N}T_{\pm,1}(N,L) + M^{-2012}
\end{align}
where $T_{\pm,1}(N,L)$ is given by
\begin{align*}
\frac{N}{\sqrt{M}}&\mathop{\sum\sum}_{n_1,n_2=1}^\infty\frac{|\lambda(n_2, n_1)|}{n_1n_2}U\left(\frac{n_1^2n_2}{L}\right)\Bigl|\sum_{\substack{q=1\\(q,M_1)=1\\n_1|q}}^{\infty}\;\sum_{(m,q)=1}\frac{ \bar\chi_2(m)\chi(q)}{a(m,q)}\mathcal B(n_1,\pm n_2,m,q)\mathcal I_\pm( n_1^2n_2,m;q)\Bigr|
\end{align*}
with $U$ a compactly supported smooth bump function on $(0,\infty)$ satisfying $U(y)=1$ for $y\in [1,2]$ and $U^{(j)}(y)\ll_j 1$, and
$$
\mathcal B(n_1,n_2,m,q):=S(M_2\overline{mM_1}, n_2\overline{M_1}; \hat q)\left[\mathop{\sideset{}{^{\star}}\sum}_{ b\bmod{M_1}}\bar\chi_1(\overline{M_2}m-b)S(\overline{b\hat q}, n_2\hat q; M_1)\right].
$$\\



\section{Cauchy inequality and Poisson summation}
\label{cps}

Applying Cauchy and Lemma \ref{ram-on-av} we have
\begin{align}
\label{tpm}
T_{\pm,1}(N,L)\ll \frac{N}{\sqrt{M}}\sqrt{\mathcal T_{\pm,1}(N,L)}
\end{align}
where
\begin{align*}
\mathcal T_{\pm,1}(N,L)=\mathop{\sum\sum}_{n_1,n_2=1}^\infty\frac{1}{n_2}U\left(\frac{n_1^2n_2}{L}\right)\Bigl|\sum_{\substack{q=1\\(q,M_1)=1\\n_1|q}}^{\infty}\;\sum_{(m,q)=1} \frac{\bar\chi_2(m)\chi(q)}{a(m,q)}\mathcal B(n_1,\pm n_2,m,q)\mathcal I_\pm( n_1^2n_2,m;q)\Bigr|^2.
\end{align*}
For notational simplicity let us only consider $\mathcal T_{+,1}(N,L)$. The other case can be analysed similarly. Opening the absolute square we get
\begin{align*}
\mathop{\sum}_{n_1=1}^\infty\mathop{\sum\sum}_{\substack{q,q'=1\\(qq',M_1)=1\\n_1|q,q'}}^{\infty}\;\mathop{\sum\sum}_{\substack{(m,q)=1\\(m',q')=1}} \frac{\bar\chi_2(m)\chi_2(m')\chi(q)\bar\chi(q')}{a(m,q)a(m',q')}T^\star(n_1,m,m',q,q')
\end{align*}
where
\begin{align*}
T^\star(n_1,m,m',q,q')=\mathop{\sum}_{n_2=1}^\infty \frac{1}{n_2}U\left(\frac{n_1^2n_2}{L}\right)\mathcal B&(n_1,n_2,m,q)\overline{\mathcal B(n_1,n_2,m',q')}\\
&\times \mathcal I_+( n_1^2n_2,m;q)\overline{\mathcal I_+( n_1^2n_2,m';q')}.
\end{align*}\\

Applying Poisson summation formula with modulus $\hat q\hat q'M_1$ (where $\hat q=q/n_1$, $\hat q'=q'/n_1$) we arrive at the following.
\begin{lemma}
\label{last-poi}
We have 
$$
T^\star(n_1,m,m',q,q')=\frac{n_1^2}{qq'M_1}\mathop{\sum}_{n_2\in\mathbb Z} C^\star(n_1,n_2,m,m',q,q')I^\star(n_2,m,m',q,q')
$$
where the character sum is given by
$$
C^\star(n_1,n_2,m,m',q,q')=\sum_{c\bmod{\hat q\hat q'M_1}}\mathcal B(n_1,c,m,q)\overline{\mathcal B(n_1,c,m',q')}e\left(\frac{cn_2}{\hat q\hat q'M_1}\right)
$$
and the integral is given by
$$
I^\star(n_2,m,m',q,q')=\int_{\mathbb R}U\left(y\right)\mathcal I_+(Ly,m;q)\overline{\mathcal I_+(Ly,m';q')}e\left(-\frac{n_2Ly}{qq'M_1}\right)y^{-1}\mathrm{d}y.
$$\\
\end{lemma}

The integral can be analysed using integration by parts and Lemma \ref{int-trns-bd}.
\begin{lemma}
\label{integral-lemma}
The integral $I^\star(n_2,m,m',q,q')$ is negligibly small if $|m|\gg M^{1+\varepsilon}Q/N$  or if $|m'|\gg M^{1+\varepsilon}Q/N$ or if $|n_2|\gg Q^2M_1M^\varepsilon/L$. Otherwise we have the bound
$$
I^\star(n_2,m,m',q,q')\ll \frac{LN}{qq'QM_1^3}M^\varepsilon.
$$ 
\end{lemma}
\begin{proof}
The first half of the first statement follows from the first statement of Lemma \ref{int-trns-bd}. That the integral is negligibly small for $|n_2|\gg Q^2M_1M^\varepsilon/L$ follows from repeated integration by parts and the last bound from Lemma \ref{int-trns-bd}. For the second statement we use the second statement of Lemma \ref{int-trns-bd} and evaluate the integral trivially.

\end{proof}


\section{Proof of Proposition \ref{prop1}}
\label{proof}

In the last section we will analyse the character sum which appears in Lemma~\ref{last-poi}.  The following lemma, which essentially gives square root cancellation in the generic case, is a consequence of Lemma~\ref{astar-lemma} and Lemma~\ref{bstar}. 
\begin{lemma}
\label{char-sum-lemma}
For $n_2\neq 0$ we have 
$$
C^\star(n_1,n_2,m,m',q,q')\ll \hat q\hat q'(\hat q,\hat q', n_2)M_1^{5/2}(M_1,n_2, m\hat q^2-m'(\hat q')^2)^{1/2},
$$
and for $n_2=0$ the sum vanishes unless $\hat q=\hat q'$ (i.e. $q=q'$) in which case
$$
C^\star(n_1,0,m,m',q,q)\ll \hat q^2c_{\hat q}(m-m')M_1^{5/2}(M_1, m-m')^{1/2},
$$
where 
$$
c_\ell(u)=\sideset{}{^\star}\sum_{\alpha\bmod{\ell}}e\left(\frac{u\alpha}{\ell}\right)
$$
is the Ramanujan sum.\\
\end{lemma}

Assuming this lemma, we can now finish the proof of Proposition \ref{prop1}. The contribution of the zero frequency $n_2=0$ to $\mathcal T_{+,1}(N,L)$ (as given in \eqref{tpm}) is bounded by
\begin{align*}
&M^\varepsilon\frac{M_1^{3/2}}{Q^2}\mathop{\sum}_{\substack{q\leq Q\\(q,M_1)=1}}\sum_{n_1|q} \;\mathop{\sum\sum}_{|m|, |m'| \ll M^{1+\varepsilon}Q/N} c_{\hat q}(m-m')(M_1, m-m')^{1/2}|I^\star(0,m,m',q,q)|+M^{-2012}\\
\ll &M^\varepsilon\frac{LM}{Q^2M_1^{3/2}}\mathop{\sum}_{\substack{q\ll Q\\(q,M_1)=1}}\frac{1}{q^2}\sum_{n_1|q}\;\left\{\frac{q\sqrt{M_1}}{n_1}+\frac{QM}{N}\right\}\ll M^\varepsilon\left\{\frac{LM}{Q^2M_1}+\frac{LM^2}{QNM_1^{3/2}}\right\}.
\end{align*}
The first term on the right hand side is the diagonal contribution $m=m'$, and is of larger size than the other term if $N>M_2^2$ (which is the case we are dealing with). (See the statement after Proposition~\ref{prop1}.) The contribution of the non-zero frequency $n_2\neq 0$ to $\mathcal T_{+,1}(N,L)$ is bounded by
\begin{align*}
&M^\varepsilon\frac{M_1^{3/2}}{Q^2}\mathop{\sum}_{n_1\ll Q}\mathop{\sum\sum}_{\substack{q,q'\leq Q\\(qq',M_1)=1\\n_1|q,q'}}\;\mathop{\sum\sum}_{0<|m|,|m'|\ll \frac{M^{1+\varepsilon}Q}{N}}\; \sum_{0<|n_2|\ll \frac{M_1Q^2M^{\varepsilon}}{L}} (qM_1, n_2)|I^\star(n_2,m,m',q,q')|+M^{-2012}\\
&\ll M^\varepsilon\frac{M_1^{3/2}}{Q^2}\frac{LN}{QM_1^3}\mathop{\sum}_{n_1\ll Q}\mathop{\sum\sum}_{\substack{q,q'\leq Q\\(qq',M_1)=1\\n_1|q,q'}}\frac{1}{qq'}\left(\frac{MQ}{N}\right)^2 \frac{M_1Q^2}{L} \ll M^\varepsilon\frac{QM^2}{N\sqrt{M_1}}.
\end{align*}
Hence it follows (since $Q^2=N/M_1$ and $L\ll Q^3M_1^3M^\varepsilon/N$) that (see \eqref{tpm})
\begin{align*}
T_{+,1}(N,L)\ll M^\varepsilon\frac{N}{\sqrt{M}}\left\{\sqrt{\frac{LM}{Q^2M_1}}+\sqrt{\frac{QM^2}{N\sqrt{M_1}}}\right\}\ll M^\varepsilon\frac{N^{5/4}}{\sqrt{N}}\left\{M_1^{3/4}+\sqrt{M_2}\right\}.
\end{align*}
The same bound holds for $T_{+,1}(N)$ by \eqref{T+}. This completes the proof of Proposition \ref{prop1}.


\section{Character sums}
\label{char}

In this section we will estimate the character sums. Using the coprimality $(qq',M_1)=1$ we see that the character sum $C^\star(n_1,n_2,m,m',q,q')$ factorizes as a product of 
$$
A^\star=\sum_{c\bmod{\hat q\hat q'}}S(M_2\overline{mM_1}, c; \hat q)S(M_2\overline{m'M_1}, c; \hat q')e\left(\frac{cn_2}{\hat q\hat q'}\right)
$$
and
$$
B^\star=\sum_{c\bmod{M_1}}\mathop{\sideset{}{^{\star}}\sum}_{ b\bmod{M_1}}\bar\chi_1(\overline{M_2}m-b)S(\overline{b\hat q}, c \bar{\hat q}; M_1)\mathop{\sideset{}{^{\star}}\sum}_{ b'\bmod{M_1}}\chi_1(\overline{M_2}m'-b')S(\overline{b'\hat q'}, c \overline{\hat {q}'}; M_1)e\left(\frac{c\overline{\hat q\hat q'}n_2}{M_1}\right).
$$\\

\begin{lemma}
\label{astar-lemma}
We have
$$
A^\star\ll \hat q\hat q'(\hat q,\hat q', n_2).
$$
Moreover for $n_2=0$ we get that $A^\star=0$ unless $\hat q=\hat q'$, in which case we get
$$
A^\star=\hat q\hat q'c_{\hat q}(m-m').
$$
\end{lemma}
\begin{proof} 
Let $p$ be a prime, $\hat q=p^jr$ and $\hat q'=p^kr'$ with $p\nmid rr'$. The $p$-part of $A^\star$ is given by
$$
A_p^\star=\sum_{c\bmod{p^{j+k}}}S(M_2\overline{mM_1r}, c\bar r; p^j)S(M_2\overline{m'M_1r'}, c\bar r'; p^k)e\left(\frac{c\overline{rr'}n_2}{p^{j+k}}\right).
$$ 
Opening the Kloosterman sums we get
\begin{align*}
A_p^\star=&\mathop{\sideset{}{^\star}\sum\sideset{}{^\star}\sum}_{\substack{a\bmod{p^j}\\a'\bmod{p^k}}}e\left(\frac{aM_2\overline{mM_1r}}{p^j}+\frac{a'M_2\overline{m'M_1r'}}{p^k}\right)\sum_{c\bmod{p^{j+k}}}e\left(\frac{c\overline{ra}p^k+c\overline{r'a'}p^j+c\overline{rr'}n_2}{p^{j+k}}\right)\\
=&p^{j+k}\mathop{\sideset{}{^\star}\sum_{a\bmod{p^j}}\;\sideset{}{^\star}\sum_{a'\bmod{p^k}}}_{\overline{ra}p^k+\overline{r'a'}p^j+\overline{rr'}n_2\equiv 0\bmod{p^{j+k}}}e\left(\frac{aM_2\overline{mM_1r}}{p^j}+\frac{a'M_2\overline{m'M_1r'}}{p^k}\right).
\end{align*}
The last sum vanishes unless $(p^j,p^k)|n_2$, and in this case we get that
$$
A_p^\star\ll p^{j+k}(p^j,p^k)\ll p^{j+k}(p^j,p^k, n_2).
$$
The lemma follows. \end{proof}

Next we will estimate $B^\star$. This sum is much more complicated and we need to use deep results of Deligne (as developed in \cite{DL} and \cite{F}). 

\begin{lemma}
\label{bstar}
We have
$$
B^\star \ll  M_1^{5/2}\sqrt{(M_1,n_2,m\hat q^2- m'(\hat q')^2)}.
$$
\end{lemma}

\begin{proof}
Opening the Kloosterman sums and executing the sum over $c$ we arrive at
$$
B^\star=M_1\mathop{\sideset{}{^{\star}}\sum_{ a\bmod{M_1}}\;\sideset{}{^{\star}}\sum_{ a'\bmod{M_1}}\;\sideset{}{^{\star}}\sum_{ b\bmod{M_1}}\;\sideset{}{^{\star}}\sum_{ b'\bmod{M_1}}}_{a\hat q+a'\hat q'+aa'n_2\equiv 0\bmod{M_1}}\bar\chi_1(\overline{M_2}m-b)\chi_1(\overline{M_2}m'-b')e\left(\frac{a\overline{b\hat q}+a'\overline{b'\hat q'}}{M_1}\right).
$$
For $n_2\equiv 0 \bmod{M_1}$ we get $a'\equiv -a\hat q\overline{\hat q'} \bmod{M_1}$ and it follows that
$$
B^\star=M_1^2\sideset{}{^{\star}}\sum_{ b\bmod{M_1}}\bar\chi_1(\overline{M_2}m-b)\chi_1(\overline{M_2}m'-b(\hat q\overline{\hat q'})^2)- M_1\mathop{\sideset{}{^{\star}}\sum\sideset{}{^{\star}}\sum}_{\substack{b\bmod{M_1}\\b'\bmod{M_1}}}\bar\chi_1(\overline{M_2}m-b)\chi_1(\overline{M_2}m'-b').
$$
The last double sum is clearly bounded by $O(M_1)$. The other sum has no cancellation if $m\hat q^2\equiv m'(\hat q')^2\bmod{M_1}$ and we get $B^\star\ll M_1^3$, otherwise we have square-root cancellation 
$$
B^\star\ll M_1^{5/2}.
$$

Next suppose $M_1\nmid n_2$. We see that $\hat q+a'n_2$ has to be invertible modulo $M_1$. Set $\gamma = \overline{\hat q+a'n_2}$ so that
$$
a'=\bar n_2(\bar\gamma-\hat q),\;\;\;\text{and}\;\;\;a=-\hat q'\bar n_2(1-\hat q\gamma).
$$
Then we have
\begin{align*}
B^\star=&M_1\sideset{}{^{\star}}\sum_{\substack{\gamma\bmod{M_1}\\M_1\nmid \gamma-\bar{\hat q}}}\;\sideset{}{^{\star}}\sum_{ b\bmod{M_1}}\;\sideset{}{^{\star}}\sum_{ b'\bmod{M_1}}\bar\chi_1(\overline{M_2}m-b)\chi_1(\overline{M_2}m'-b')e\left(\frac{-\overline{b\hat qn_2}\hat q'(1-\hat q\gamma)+\overline{b'\hat q'n_2}(\bar\gamma-\hat q)}{M_1}\right).
\end{align*}
The sum over $\gamma$ can now be extended over $\mathbb F_{M_1}^\star$. The extra term corresponding to $\gamma=\bar{\hat q}$ makes a contribution of order $O(M_1)$. We get that
$$
B^\star=M_1S_{M_1}(f_1,f_2;g)+O(M_1)
$$
where 
$$
f_1:=\overline{M_2}m-x_1,\;\; f_2:=\overline{M_2}m'-x_2,\;\; g:=-\overline{\hat qn_2}\hat q'x_1^{-1}(1-\hat qx_3)+\overline{\hat q'n_2}x_2^{-1}(x_3^{-1}-\hat q)
$$
are Laurent polynomials in $\mathbb F_{M_1}[x_1,x_2,x_3,(x_1x_2x_3)^{-1}]$ and 
$$
S_{M_1}(f_1,f_2;g):=\sum_{\mathbf{x}\in (\mathbb F_{M_1}^\star)^3}\bar\chi_1(f_1(\mathbf{x}))\chi_1(f_2(\mathbf{x}))e\left(\frac{g(\mathbf{x})}{M_1}\right).
$$
Such mixed character sums have been studied in \cite{F} (following the method of \cite{DL}). In particular one has square root cancellation in the sum once the Laurent polynomial
$$
F(x_1,\dots,x_5):=g(x_1,x_2,x_3)+x_4f_1(x_1,x_2,x_3)+x_5f_2(x_1,x_2,x_3)
$$
is non-degenerate with respect to its Newton polyhedra $\Delta_\infty(F)$. Non-degeneracy in the special case under consideration can be checked quite easily. The lemma follows.
\end{proof}


\end{document}